\documentclass[11pt]{article}
\usepackage{etex}
\usepackage[a4paper]{anysize}\marginsize{1.5cm}{1.5cm}{1.5cm}{1cm}
\pdfpagewidth=\paperwidth \pdfpageheight=\paperheight
\usepackage{amsfonts,amssymb,amsthm,amsmath,eucal,tabu,url}
\usepackage{array}
\usepackage{pstricks}
\usepackage{pstricks-add}
\usepackage{pgf,tikz}
\usepackage{indentfirst}
\usepackage[small]{diagrams}
\diagramstyle[centredisplay,nohug]

\theoremstyle{plain}
\newtheorem{thm}{Theorem}[section]
\newtheorem{theorem}[thm]{Theorem}

\newtheorem*{theoremA}{Theorem A}
\newtheorem*{theoremB}{Theorem B}

\newtheorem{lemma}[thm]{Lemma}

\newtheorem{corollary}[thm]{Corollary}
\newtheorem{conjecture}[thm]{Conjecture}

\theoremstyle{definition}
\newtheorem{definition}[thm]{Definition}
\newtheorem{remark}[thm]{Remark}

\newtheorem{thevarthm}[thm]{\varthmname}

\newenvironment{varthm*}[1]{\trivlist\item[]{\bf #1.}\it}{\endtrivlist}

\renewcommand\geq{\geqslant}

\renewcommand\leq{\leqslant}

\newcommand\be{\begin{eqnarray*}}
\newcommand\ee{\end{eqnarray*}}

\newcommand\Q{\mathbb Q}

\newcommand\Z{\mathbb Z}

\renewcommand\P{\mathbb P}

\newcommand\calo{{\mathcal O}}

\newcommand\newop[2]{\def#1{\mathop{\rm #2}\nolimits}}
\newop\log{log}
\newop\ord{ord}
\newop\Gal{Gal}
\newop\SL{SL}
\newop\Bl{Bl}
\newop\mult{mult}
\newop\mass{mass}
\newop\div{div}
\newop\codim{codim}
\newop\sing{sing}
\newop\vdim{vdim}
\newop\edim{edim}
\newop\Ass{Ass}
\newop\size{size}
\newop\reg{reg}
\newop\satdeg{satdeg}
\newop\supp{supp}
\newop\Neg{Neg}
\newop\Nef{Nef}
\newop\Nefh{Nef_H}
\newop\Eff{Eff}
\newop\Zar{Zar}
\newop\MB{MB}
\newop\MBxC{MB\mathit{(x,C)}}
\newop\NnB{NnB}
\newop\Bigg{Big}
\newop\Sing{Sing}
\newop\Esing{Esing}
\newop\Effbar{\overline{\Eff}}

\newcommand\eqnref[1]{(\ref{#1})}

\newcommand\wtilde[1]{\widetilde{#1}}

\def\keywordname{{\bfseries Keywords}}%
\def\keywords#1{\par\addvspace\medskipamount{\rightskip=0pt plus1cm
\def\and{\ifhmode\unskip\nobreak\fi\ $\cdot$
}\noindent\keywordname\enspace\ignorespaces#1\par}}
\def\subclassname{{\bfseries Mathematics Subject Classification
(2000)}\enspace}
\def\subclass#1{\par\addvspace\medskipamount{\rightskip=0pt plus1cm
\def\and{\ifhmode\unskip\nobreak\fi\ $\cdot$
}\noindent\subclassname\ignorespaces#1\par}}

\begin{document}
\title{Bounded negativity, Harbourne constants and transversal arrangements of curves}
\author{Piotr Pokora, Xavier Roulleau, Tomasz Szemberg}
\date{\today}
\maketitle
\thispagestyle{empty}
\begin{abstract}
   The Bounded Negativity Conjecture predicts that for every complex projective
   surface $X$ there exists a number $b(X)$ such that $C^2\geq -b(X)$ holds
   for all reduced curves $C\subset X$.
   For birational surfaces $f:Y\to X$ there have been introduced in \cite{BdRHHLPSz}
   certain invariants (Harbourne constants) relating to the effect the numbers $b(X)$,
   $b(Y)$ and the complexity of the map $f$. These invariants have been
   studied when $f$ is the blowup of all singular points of an arrangement
   of lines in $\P^2$ in \cite{BdRHHLPSz}, of conics in \cite{PTG}
   and of cubics in \cite{XR}. In the present note we extend these
   considerations to blowups of $\P^2$ at singular points of arrangements
   of curves of arbitrary degree $d$. The main result in this direction
   is stated in Theorem B. We also considerably generalize and modify
   the approach witnessed so far and study transversal arrangements of
   sufficiently positive curves on arbitrary surfaces with the non-negative
   Kodaira dimension. The main result obtained  in this general setting
   is presented in Theorem A.
\subclass{14C20, 14J70}
\end{abstract}

%*****************************************************************************
\section{Introduction}
   In this note we find various estimates on Harbourne constants which were introduced in \cite{BdRHHLPSz} in order to capture and measure
   the bounded negativity on various birational models of an algebraic surface. Our research is motivated by
   Conjecture \ref{conj:BNC} below which is  related to the following definition:
\begin{definition}[Bounded negativity]\label{def:has BN}
   Let $X$ be a smooth projective surface. We say that $X$ \emph{has bounded negativity}
   if there exists an integer $b(X)$ such that the inequality
   $$C^{2} \geq -b(X)$$
   holds for every \emph{reduced and irreducible} curve $C \subset X$.
\end{definition}
   The bounded negativity conjecture (BNC for short) is one of the most intriguing problems in the theory of projective surfaces and
   attracts currently a lot of attention, see e.g. \cite{Duke, BdRHHLPSz, Harbourne1, XR, Dor16}.
\begin{conjecture}[BNC]\label{conj:BNC}
   Every smooth \emph{complex} projective surface has bounded negativity.
\end{conjecture}
%\begin{remark}
   It is well known that Conjecture \ref{conj:BNC} fails in positive characteristic.
   Hence from now on we restrict the attention to complex surfaces.
%\end{remark}

   It has been showed in \cite[Proposition 5.1]{Duke} that no harm is done if one replaces irreducible curves
   in Definition \ref{def:has BN} by arbitrary \emph{reduced} divisors. It is clear that in order to obtain interesting,
   i.e. very negative curves on the blow up of a given surface one should study singular curves on the
   original surface. Whereas constructing irreducible singular curves encounters a number of obstacles
   (see e.g. \cite{GLS}), reducible singular divisors are relatively easy to construct and control.
   In our set up singularities of reduced divisors arise solely as intersection points of irreducible
   components.
   In a series of papers \cite{BdRHHLPSz, PTG,XR} the authors study this situation
   for configurations of lines, conics and elliptic curves in $\P^2$.
   The arrangements studied so far were all modeled on arrangements of lines,
   in particular all curves were smooth and were assumed to
   intersect pairwise transversally.
   The technical advantage behind this assumption lies in the property
   that after blowing up all intersection points just once, we obtain
   a simple normal crossing divisor.
   Also working under this assumption for curves of higher
   degree seems to lead to the most singular divisors. Many singularities of a divisor
   lead to its negative arithmetic genus, which forces the divisor to split.
   Moreover, transversal arrangements allow to use some combinatorial identities,
   which fail when tangencies are allowed. For all these reasons it is reasonable to keep this assumption.
\begin{definition}[Transversal arrangement]\label{def:transversal arrangement}
   Let $D=\sum_{i=1}^{\tau}C_i$ be a reduced divisor on a smooth surface $X$.
   We say that $D$ is a \emph{transversal arrangement} if $\tau\geq 2$, all curves $C_i$ are smooth
   and they intersect pairwise transversally.\\
   We denote by $\Sing(D)$ the set of all intersection points
   of components of $D$. The number of points in the set $\Sing(D)$ is denoted by $s(D)$ or, if $D$
   is understood, simply by $s$.\\
   Furthermore we denote by $\Esing(D)$ the set of \emph{essential singularities} of $D$,
   i.e. those where at least $3$ components meet.
\end{definition}
%   The above definition has an important geometric consequence.
%   \begin{remark}
%      Let $D$ be a arrangement of smooth curves on $X$ having only transversal intersection points and let $f:Y\to X$ be the blow up
%      of $X$ in the the set $\Sing(D)$ with exceptional divisors
%      $E_1,\ldots,E_s$. Let $\wtilde{D}$ be the proper transform of $D$.
%      Then the divisor
%      $$\wtilde{D}+E_1+\ldots+E_s$$
%      is a simple normal crossing (SNC for short) divisor on $Z$.
%   \end{remark}
   In the present note we study the bounded negativity and transversal arrangements
   on \emph{fairly arbitrary} surfaces. Our main results are Theorems A and B.
\begin{theoremA}
   Let $A$ be a divisor on a smooth projective surface $Y$ with Kodaira dimension $\kappa(Y) \geq 0$,
   such that for positive integers $\tau\geq 2$ and $d_{1},\ldots, d_{\tau}\geq 1$ the following condition is satisfied:

   $(\star)$ There exist smooth (irreducible) curves $C_{1}, ..., C_{\tau}$ in linear systems $|d_{1}A|, ..., |d_{\tau}A|$
   such that the divisor $D = \sum_{i=1}^{\tau} C_{i}$ is a transversal arrangement.

   Let $f : Z \rightarrow Y$ be the blow-up of $Y$ at ${\rm Sing}(D)$ and denote by $\wtilde{D}$ the strict transform of $D$.
   Then
   $$\wtilde{D}^{2} \geq -\frac{9}{2}s - \left(\frac{3}{2}A^{2}\sum_{i=1}^{\tau} d_{i}^{2}+(K_{Y}\cdot A)\sum_{i=1}^{\tau} d_{i}+2(3c_{2}(Y)-c_1^2(Y))\right).$$
\end{theoremA}
   The assumption $\kappa(Y)\geq 0$ guarantees that any finite branched covering
   of $Y$ has also non-negative Kodaira dimension. If we can control the Kodaira dimension of a covering of $Y$
   in other way, then we can drop this assumption. This is the case in the next Theorem which address
   rational surfaces. Recently Dorfmeister \cite{Dor16} has announced
   a proof of Conjecture \ref{conj:BNC} for surfaces birationally equivalent to ruled surfaces
   (i.e. in particular for rational surfaces). This announcement has been taken back in the last days.
   Whereas this
   would be an exciting new development, it would not diminish the interest in effective bounds on Harbourne constants.
\begin{theoremB}
   Let $D \subset \mathbb{P}^{2}$ be a transversal arrangement of $\tau \geq 4$ curves $C_1,\ldots,C_\tau$ of degree $d \geq 3$
   such that there are no points in which all $\tau$ curves meet, i.e. the linear series spanned by $C_1,\ldots,C_\tau$
   is base point free.
   Let $f : X_{s} \rightarrow \mathbb{P}^{2}$ be the blowup at $\Sing(D)$
   and let $\wtilde{D}$ be the strict transform of $D$.
   Then we have
   $$\wtilde{D}^2 \geq \frac{9d\tau}{2}-\frac{5d^2\tau}{2}-4s.$$
\end{theoremB}
   This result provides additional evidence for the following effective version
   of Conjecture \ref{conj:BNC} which predicts that there are uniform bounds for all blow ups of $\P^2$.
\begin{conjecture}[Effective BNC for blowups of $\P^2$]\label{effectiveBNC}
    Let $f : X_{s} \rightarrow \mathbb{P}^{2}$ the blow up of $\mathbb{P}^{2}$ in $s$ arbitrary points.
    Let $D \subset \mathbb{P}^{2}$ be a reduced divisor and
    let $\wtilde{D}$ be the strict transform of $D$ under $f$.
    Then one has $\wtilde{D}^{2} \geq -4 \cdot s$.
\end{conjecture}
   Our strategy is an extension of Hirzerbuch's results \cite{Hirzebruch}
   for line configurations on the plane.
   The starting point is that (under some conditions) one can construct an abelian cover $W$
   of the studied surface branched along the chosen configurations of curves.
   If the singularities of these configurations are reasonable (simple crossings),
   the Chern numbers of that abelian cover (or rather its minimal resolution $X$) can be explicitly computed,
   and it turns out that these Chern numbers can be read off directly from combinatorics of the given configuration.
   Moreover, under some additional mild assumptions on multiplicities of singular points of the configuration,
   the surface $X$ is of general type. The last step is made by the Miyaoka-Yau inequality
   $K_{X}^{2} \leq 3e(X)$, which gives us the inequalities of Theorems A and B.

%*****************************************************************************
\section{General preliminaries}
   We begin by introducing some invariants of transversal arrangements and pointing
   out their properties relevant for our purposes in this note.
\begin{definition}[Combinatorial invariants of transversal arrangements]
   Let $D=\sum_{i=1}^\tau C_i$ be a transversal arrangement on a smooth surface $X$.
   We say that a point $P$ is an $r$-fold point of the arrangement $D$ if there
   are exactly $r$ components $C_i$ passing through $P$. We say also that $D$ has multiplicity $k_P=r$ at $P$.\\
   For $r\geq 2$ we set the numbers $t_r=t_r(D)$ to be the number of $r$-fold points in $D$.
   Thus $s(D)=\sum_{r=2}^\tau t_r(D)$.
\end{definition}
   These numbers are subject to the following useful equality, which follows by
   counting incidences in a transversal arrangement in two ways.
\begin{equation}\label{eq:combinatorial general}
   \sum_{i<j}(C_i\cdot C_j)=\sum_{r\geq 2}\binom{r}{2}t_r.
\end{equation}
   It is also convenient to introduce the following numbers
   $$f_i=f_i(D)=\sum_{r\geq 2}r^i t_r.$$
   In particular $f_0=s(D)$ is the number of points in $\Sing(D)$.

   Now we turn to Harbourne constants. They were first discussed at the
   Negative Curves on Algebraic Surfaces workshop in Oberwolfach in spring 2014 and were introduced
   in the literature as Hadean constants in \cite{BdRHHLPSz}.
   In the present note we are interested in Harbourne constants attached to transversal
   arrangements. They can be viewed as a way to measure the average negativity coming
   from singular points in the arrangement.
\begin{definition}[Harbourne constants of a transversal arrangement]\label{def:H-constant TA}
   Let $X$ be a smooth projective surface.
   Let $D=\sum_{i=1}^\tau C_i$ be a transversal arrangement of curves on $X$ with $s=s(D)$.
   The rational number
   \begin{equation}\label{eq:TA Harbourne constant}
      h(X;D)=h(D)=\frac1s\left(D^2-\sum\limits_{P\in \Sing(D)}k_P^2\right)
   \end{equation}
   is the \emph{Harbourne constant of the transversal arrangement} $D\subset X$.
%   The real number
%   $$H_{ta}(X)=\inf_D h(X;D),$$
%   where the infimum is taken over all transversal arrangements $D\subset X$
%   is the \emph{global transversal arrangements Harbourne constant} of $X$.
\end{definition}
   The connection between Harbourne constants and the BNC is established by the following observation.
   If the Harbourne constants $h(X;D)$ (here we mean Harbourne constants for all curve configurations) on the fixed surface $X$ are uniformly bounded from below by a number $H$,
   then BNC holds for all birational models $Y=\Bl_{\Sing(D)}X$ obtained from $X$ by blowing up
   singular points of transversal arrangements $D$ with $b(Y) = H\cdot s(D)$.
   The reverse implication might fail, i.e.
   it might happen that there is no uniform lower bound but nevertheless BNC may hold
   on any single model of $X$.

   In case of the projective plane it is convenient to work with a more
   specific variant of Definition \ref{def:H-constant TA}.
   In \cite[Definition 3.1]{BdRHHLPSz} the authors introduced the
   \emph{linear Harbourne constant} as the infimum of quotients
   in \eqnref{eq:TA Harbourne constant}, where one considers only divisors $D$ splitting
   totally into lines. In \cite{PTG} the {conical Harbourne constant}
   has been studied and in \cite{XR} the {cubical Harbourne constant} has
   been considered. Here we follow this line of investigation and introduce the
   following notion.
\begin{definition}[Degree $d$ Harbourne constant]\label{def:degree d Har Cons}
   The \emph{degree $d$ global Harbourne constant of $\P^2$} is the infimum
   $$H_d(\P^2):=\inf_{D} h(\P^2;D),$$
   taken over all transversal arrangements $D$ of degree $d$ curves in $\P^2$.
\end{definition}
   We will show in Section \ref{sec:deg d in P2} bounds on the degree $d$ Harbourne constants $H_d(\P^2)$ for arbitrary $d\geq 3$.
   The available bounds on the numbers $H_d(\P^2)$ are presented in Table \ref{tab:H constants}.
\begin{table}[h]
\centering{
\renewcommand{\arraystretch}{1.3}
\begin{tabular}{|c|c|c|}
  \hline
  $d$ & lower bound on $H_d(\P^2)$ & least known value of $H_d(\P^2)$\\
  \hline
  1 & -4 & -225/67 \\
  2 & -4.5 & -225/68 \\
  \hline
\end{tabular}}
\caption{: degree $d$ global Harbourne constants}
\label{tab:H constants}
\end{table}
   In the article \cite{XR} there is studied a series of configurations of
   smooth elliptic plane curves with Harbourne constants tending to $-4$.
   These configurations are not transversal (there are always $12$ points
   where configuration curves are pairwise tangential).
   The following result is derived from Theorem B and, to the best of our
   knowledge, this is the first effective estimate on
   degree $d$ Harbourne constants.
\begin{corollary}[Degree $d$ Harbourne constants]\label{cor:deg d HC}
   For any $d\geq 3$ we have
   $$H_d(\P^2)\geq \frac92d-\frac52d^2-4.$$
\end{corollary}
\begin{remark}
   Whereas the particular numbers appearing in Corollary \ref{cor:deg d HC} are rather high and
   leave space for improvements, the main interest of the Corollary lies
   in the conclusion that they are \emph{finite} (which is by no means a priori obvious)
   and can be estimated effectively.
\end{remark}
\section{Bounded negativity and transversal arrangements on surfaces with Kodaira dimension $\kappa \geq 0$}
   In this section we will prove Theorem A. In fact, we will prove slightly more.
   We establish first the notation.
   Let $Y$ be a smooth projective surface and let $A$ be a semi-ample divisor on $Y$.
   We assume moreover that the following hypothesis holds for
   $A$ and for integers $d_{1,}\dots,d_{\tau}\in\mathbb N$, $\tau>1$ :
\begin{itemize}
   \item There exist smooth (irreducible) curves $C_{1}, ..., C_{\tau}$ in linear systems $|d_{1}A|, ..., |d_{\tau}A|$
      such that the divisor $D = \sum_{i=1}^{\tau} C_{i}$ is a transversal arrangement.
   \item We assume moreover that either all numbers $d_i$ are even, or there exist at least two odd numbers among them.
\end{itemize}
%   Note that the conditions on $A$ imply that $A^{2} \geq 0$. If additionally $A^{2} > 0$ it follows
%   that some multiple of the linear system $|A|$ is base point.
   It is convenient to write now the equality \eqnref{eq:combinatorial general}
   in the following form
   \begin{equation}\label{eq:combinatorial equality1}
       A^{2}\left(\sum d_{i}\right)^{2}-A^{2}\sum d_{i}^{2}=f_{2}-f_{1}.
   \end{equation}
   As a consequence we get
   \begin{equation}\label{eq:D^2}
      D^{2}=A^2\left(\sum d_{i}^{2}+2\sum_{j<k}d_{j}d_{k}\right)=A^2\left(\sum d_{i}^{2}+f_{2}-f_{1}\right).
   \end{equation}
\begin{theorem}\label{thm:Thm A variant}
   Let $Y$ be a smooth projective surface with Kodaira dimension $\kappa \geq 0$. Let $A$
   be a divisor on $Y$ satisfying above assumptions and let $D=\sum_{i=1}^\tau C_i$ be
   a transversal arrangement as above. Then
   $$
   H(Y;D) \geq-\frac{9}{2}+\frac{1}{f_{0}}\left(2t_{2}+\frac{9}{8}t_{3}+\frac{1}{2}t_{4}-\frac{3}{2}A^{2}\sum d_{i}^{2}-
   (K_{Y}\cdot A)\sum d_{i}-2(3c_{2}(Y)-c_1^2(Y))\right).
   $$
\end{theorem}

Our strategy for proving this statement will be to apply the refined Miyaoka inequality to a certain branched covering $X$, of $Y$.
In order to prove that  this branched covering does in fact exists, we need to recall some result of Namba:
 Let $M$ be a manifold, let $D_{1},\dots,D_{s}$ be irreducible reduced divisors
on $M$ and let $n_{1},\dots,n_{s}$ be positive integers. We denote
by $D$ the divisor $D=\sum n_{i}D_{i}$. Let $Div(M,D)$ be the sub-group
of the $\mathbb{Q}$-divisors generated by the entire divisors and:
\[
\frac{1}{n_{1}}D_{1},\dots,\frac{1}{n_{s}}D_{s}.
\]
Let $\sim$ be the linear equivalence in $Div(M,D)$, where $G\sim G'$
if and only if $G-G'$ is an entire principal divisor. Let $Div(M,D)/\sim$
be the quotient and let $Div^{0}(M,D)/\sim$ be the kernel of the
Chern class map
\[
\begin{array}{ccc}
Div(M,D)/\sim & \to & H^{1,1}(M,\mathbb{R})\\
G & \to & c_{1}(G)
\end{array}.
\]

\begin{thm}
\label{thm:(Namba).-There-exists}(Namba, \cite[Theorem 2.3.20]{Namba}).
There exists a finite Abelian cover which branches at D with index
$n_{i}$ over $D_{i}$ for all $i=1,\dots,s$ if and only if for every
$j=1,\dots,s$ there exists an element of finite order $v_{j}=\sum\frac{a_{ij}}{n_{i}}D_{i}+E_{j}$
of $Div^{0}(M,D)/\sim$ (where $E_{j}$ an entire divisor and $a_{ij}\in \mathbb{Z}$)
such that $a_{jj}$ is coprime to $n_{j}$. \\
Then the subgroup in $Div^{0}(M,D)/\sim$ generated by the $v_{j}$
is isomorphic to the Galois group of such an Abelian cover.
\end{thm}

Let us now recall the following combination of results due to Miyaoka \cite{Miyaoka1} and Sakai \cite{Sakai} which was formulated in this form for the first time by Hirzebruch.
\begin{theorem}(Miyaoka-Sakai refined inequality \cite[p.~144]{Hirzebruch86}). \label{MiyaSakai}
Let $X$ be a smooth surface of non-negative Kodaira dimension and be $E_{1}, ..., E_{k}$ be configurations (disjoint to each other) of rational curves on $X$ (arising from quotient singularities) and let $C_{1}, ..., C_{p}$ be smooth elliptic curves (disjoint to each other and disjoint to the $E_{i}$). Let $c_{1}^{2}(X), c_{2}(X)$ be the Chern numbers of $X$. Then
$$3c_{2}(X) - c_{1}^{2}(X) \geq \sum_{j=1}^{p}(-C_{j}^{2}) + \sum_{i=1}^{k}m(E_{i}),$$
where the number $m(E_{i})$ depends on the configuration. For example, if $E_{i}$ is a single $(-2)$-curve, then  $m(E_{i})=\frac{9}{2}$ by \cite{Hemperly}.
\end{theorem}

\begin{proof} (of Theorem \ref{thm:Thm A variant}).
%      Renumbering the curves if necessary we may assume that the integers $d_{1},\ldots,d_{s}$ are even
%      and $d_{s+1},\ldots,d_\tau$ are odd. We define $\delta=1$ if $s< \tau$ and $\delta=0$ if $s=\tau$.
%      Let $v_{1},\ldots,v_{\tau-\delta}$ be the $\mathbb{Q}$-divisors : $v_{1}=\frac{1}{2}C_{1}$,...,$v_{s}=\frac{1}{2}C_{s}$
%      and if $s< \tau-1$ we set
%      $$ v_{s+1}=\frac{1}{2}(C_{s+1}-C_{\tau}),\dots,v_{\tau-1}=\frac{1}{2}(C_{\tau-1}-C_{\tau}).$$
    Let be $\delta=0$ if all the $d_i$'s are even, and  $\delta=1$ otherwise. We apply Theorem  \ref{thm:(Namba).-There-exists},
   %(see also \cite[Theorem 2.1]{AlePar13})
to the $\Q$-divisors $\frac{1}{2}(C_i-C_j)$ for $d_i,d_j$ odd and $\frac{1}{2}C_j$ for $d_j$ even) :
   there exists a $(\mathbb{Z}/2\mathbb{Z})^{\tau-\delta}$
   abelian cover $\sigma:W\to Y$ ramified over $D$ with order $2$.
   We denote by $\rho:X\to W$ its minimal desingularization.
   We follow the ideas of Hirzebruch \cite{Hirzebruch} for the computations of the Chern numbers of $X$.

  For a singularity point $P$ of $D$, let $k_P$ be its multiplicity. Let $\pi: Z \to Y$ be the blowup at the $f_{0}-t_{2}=\sum_{k\geq3}t_{k}$
   singularities of $D$ with multiplicities $k\geq3$. Let $\wtilde{D}=\sum\wtilde{C}_{i}$
   be the strict transform of $D$ in $Z$ and let $E_P$ be the exceptional divisor
   over the point $P$. There exists a degree $2^{\tau-\delta}$
   map
   $$ f:X \to Z $$
   ramified over $Z$ with the divisor $\wtilde{D}$ as the branch locus of order $2$.

   These constructions are summarized in the diagram in Figure \ref{dia:diagram}.
\begin{figure}[h]
   \begin{diagram}
      X & \rTo^{\rho} & W\\
      \dTo^{f} &  & \dTo_{\sigma}\\
      Z & \rTo^{\pi} & Y\\
   \end{diagram}
   \caption{$ $ Maps used in the proof of Theorem \ref{thm:Thm A variant}.}
   \label{dia:diagram}
\end{figure}
%Let $r=k_P\geq 3$ be the multiplicity at $P\in \Esing(D)$.
There are $2^{\tau-\delta-k_P}$
   copies of a smooth curve $F_{P}\subset X$ over $E_{P}\subset Z$.
   The curve $F_P$  is a $(\mathbb{Z}/2\mathbb{Z})^{k_P-1}$-cover
   of $E_{P}$ ramified with index $2$ at $k_P$ intersection points of $E_P$ with $\wtilde{D}$. Thus
   $$e(F_{P})=2^{k_P-1}(2-k_P)+k_P 2^{k_P-2}=2^{k_P-2}(4-k_P).$$
   Since the Galois group of $f$ permutes these curves, we have $(F_{P})^{2}=-n^{k_P-2}$.
   If a singularity $P$ of $D$ is a double point, then $X$ is smooth over
   $P$ and the fiber of $\pi\circ f$ above $P$ has $n^{\tau-\delta-2}$ points.
   Following Miyaoka \cite[point G, page 408]{Miy08}, we define the genus $g=g(C)$ by
   \begin{equation}\label{eq:g}
      g-1= \sum_{i=1}^\tau (g_{i}-1),
   \end{equation}
   where $g_{i}$ is the genus of the irreducible component $C_i$ of $D$, hence
   $$2g_i-2=A^2d_i^2+(A\cdot K_Y)d_i.$$
   Summing up over $i$ we have
   \begin{equation}\label{eq:2g-2}
      2g-2=A^2\sum_{i=1}^\tau d_i^2+(A\cdot K_Y)\sum_{i=1}^\tau d_i.
   \end{equation}
   Similarly, using the additivity of the topological Euler numbers and \eqnref{eq:g} we have
   $$e(D)=2-2g+f_{0}-f_{1}$$
   and consequently
   \begin{equation}\label{eq:e1}
      e(D\setminus \Sing(D))=2-2g-f_{1}\;\mbox{ and } e(Y\setminus D)=e(Y)-e(D)=e(Y)+2g-2+f_{1}-f_{0}.
   \end{equation}
   Using that if $U \to V$ is a  degree $n$ \'etale map one has $e(U)=ne(V)$, we obtain
   $$e\left(X\setminus \bigcup\limits_{P\in\Esing(D)} f^{-1}E_{P}\right)=
      2^{\tau-\delta}e(Y\setminus D)+2^{\tau-\delta-1}e(D\setminus \Sing(D))+2^{\tau-\delta-2}t_{2}.$$
   Combining this with \eqnref{eq:e1} we get
   $$\frac{1}{2^{\tau-\delta-2}}e\left(X\setminus \bigcup\limits_{P\in\Esing(D)}f^{-1}E_{P}\right)=
      4\left(e(Y)+2g-2+f_{1}-f_{0}\right)+2\left(2-2g-f_{1}\right)+t_{2}.$$
   Since in $X$ over each exceptional divisor $E_{P}$ in $Z$, there are $2^{\tau-\delta-k_{P}}$
   curves with Euler number $e(F_{P})$, we get
   $$e(X)=e\left(X\setminus \bigcup\limits_{P\in\Esing(D)}f^{-1}E_{P}\right)+
          \sum_{k\geq3}2^{\tau-\delta-2}(4-k)t_k = e\left(X\setminus \bigcup\limits_{P\in\Esing(D)}f^{-1}E_{P}\right)+2^{\tau-\delta-2}(4f_{0}-f_{1}-2t_{2}).$$
   Thus
   \begin{equation}\label{eq:formula for c_2(X)}
      \frac{1}{2^{\tau-\delta-2}}\cdot e(X)=4e(Y)+4g-4+f_{1}-t_{2}.
   \end{equation}
   Our purpose now is to calculate the other Chern number $c_1^2(X)=K_X^2$.
   The canonical divisor $K_X$ satisfies $K_{X}=f^*K$ for the divisor $K$ on $Z$ defined as
   $$K:=\pi^{*}K_{Y}+\sum E_{P}+
      \frac{1}{2}\left(\sum E_{P}+\pi^{*}D-\sum k_{P}E_{P}\right)=
      \sum\frac{3-k_{P}}{2}E_{P}+\frac{1}{2}\pi^{*}D+\pi^{*}K_{Y},$$ %_{P\in\Esing(D)}
   with the summation taken over all points $P\in\Esing(D)$.
   We have
   $$K^{2}=-\frac{1}{4}\sum_{k\geq3}(3-k)^{2}t_{k}+\frac{1}{4}(\sum d_{i})^{2}A^2+(K_{Y}\cdot A) \sum d_{i}+K_{Y}^{2}.$$
   Using \eqnref{eq:combinatorial equality1} we get
   $$K^2=-\frac{1}{4}\left(9f_{0}-6f_{1}+f_{2}-t_{2}\right)+\frac{1}{4}(\sum d_{i})^{2}A^2+(K_{Y}\cdot A)\sum d_{i}+K_{Y}^{2}.$$
   Thus
   \begin{equation}\label{eq:formula for c_1^2 X}
      \frac{1}{2^{\tau-\delta-2}}K_{X}^{2}=-9f_{0}+6f_{1}-f_{2}+t_{2}+(\sum d_{i})^{2}A^2+4(K_{Y}\cdot A)\sum d_{i}+4K_{Y}^{2}.
   \end{equation}
   Combining \eqnref{eq:formula for c_2(X)} and \eqnref{eq:formula for c_1^2 X} we obtain
   $$
      \frac{1}{2^{\tau-\delta-2}}(3c_{2}(X)-c_1^2(X))=$$$$
      = 4(3c_{2}(Y)-c_1^2(Y))+12(g-1)+f_{2}-3f_{1}+9f_{0}-4t_{2}-4(K_{Y}\cdot A)\sum d_{i}-(\sum d_{i})^{2}A^{2}.
   $$
   The surface $X$ contains $2^{\tau-\delta-3}t_{3}$ disjoint $(-2)$-curves
   (above the $3$-points) and it contains $2^{\tau-\delta-4}t_{4}$ elliptic
   curves (above the $4$-points), each of self-intersection $-4$.
   Since the Kodaira dimension of $Y$ is non-negative, so is that of $X$.
   We can then apply the Miyaoka-Sakai refined inequality
   and we obtain that:
   $$\frac{1}{2^{\tau-\delta-2}}\left(3c_{2}(X)-c_1^2(X)\right)\geq\frac{9}{4}t_{3}+t_{4}.$$
   This gives
   $$4(3c_{2}(Y)-K_{Y}^{2})+12(g-1)+f_{2}-3f_{1}+9f_{0}-4t_{2}-4(K_{Y}\cdot A)\sum d_{i}-A^{2}(\sum d_{i})^{2}\geq\frac{9}{4}t_{3}+t_{4}.$$
   Using \eqnref{eq:combinatorial equality1} and \eqnref{eq:2g-2} we arrive finally to the following Hirzebruch-type inequality :
\begin{equation}
   5A^{2}\sum d_{i}^{2}+2(K_{Y}\cdot A)\sum d_{i}+4(3c_{2}(Y)-c_1^2(Y))-2f_{1}+9f_{0}\geq4t_{2}+\frac{9}{4}t_{3}+t_{4}.
\end{equation}
   Since $h(Y;D)=\frac1{f_{0}}\left({A^{2}(\sum d_{i})^{2}-f_{2}}\right)=\frac1{f_{0}}({A^{2}\sum d_{i}^{2}-f_{1}}),$
   we obtain:
   \begin{equation}\label{eq:h(Y;D) geq}
   h(Y;D)\geq
     -\frac{9}{2}+\frac{1}{f_{0}}\left(2t_{2}+\frac{9}{8}t_{3}+\frac{1}{2}t_{4}-\frac{3}{2}A^{2}\sum d_{i}^{2}-(K_{Y}\cdot A)
     \sum d_{i}-2(3c_{2}(Y)-c_1^2(Y))\right).
   \end{equation}
\end{proof}
   The statement in Theorem A is now an easy corollary. Indeed, note that $f_0=s$,
   $\wtilde{D}^2=s\cdot h(Y;D)$
   and we can drop on the right hand side of \eqref{eq:h(Y;D) geq} all summands of which
   we know that they are non-negative.

   Sometimes it is more convenient to work with the following version
   of the inequality in \eqref{eq:h(Y;D) geq}, which we record for future reference.
\begin{remark}
   For any transversal arrangement $D$ we have the following inequality:
   $$-2f_{1}+9f_{0}=\sum_{k\geq2}(9-2k)t_{k}\leq5t_{2}+3t_{3}+t_{4}+\sum_{k\geq5}(4-k)t_{k}=3t_{2}+2t_{3}+t_{4}+4f_{0}-f_{1}.$$
   This yields
   $$h(Y;D)=
     \frac{A^{2}\sum d_{i}^{2}-f_{1}}{f_{0}}
     \geq-4+\frac{1}{f_{0}}\left(t_{2}+\frac{1}{4}t_{3}-4A^{2}\sum d_{i}^{2}-2(K_{Y}\cdot A)\sum d_{i}-(3c_{2}(Y)-c_1^2(Y))\right).$$
\end{remark}

\section{Configurations of degree $d$ plane curves}\label{sec:deg d in P2}
   In this part in order to abbreviate the notation it is convenient to work
   with the following modification of Definition \ref{def:transversal arrangement}.
\begin{definition}\label{def:d arrangement}
   A $d$-arrangement is a transversal arrangement of smooth plane curves of degree $d$.
\end{definition}
   For a $d$-arrangement $D$, the equality in \eqref{eq:combinatorial general} has now the following form
   \begin{equation}\label{eq:combinatorial equality}
      d^2{\tau \choose 2} = \sum_{r \geq 2} {r \choose 2} t_{r}.
   \end{equation}
   where $\tau$ is the number of irreducible components of $D$.

   Theorem B follows from the following, slightly more precise statement.

\begin{theorem}\label{thm:Hirzdegreed}
   For a $d$-arrangement $D=\sum C_i \subset \mathbb{P}^{2}$ of $\tau\geq 4$ plane curves of degree $d \geq 3$ such that $t_{\tau} = 0$
   we have
   $$h(\P^2,D) \geq -4 + \frac{-\frac52d^{2}\tau+\frac92d\tau}{s}.$$
\end{theorem}
\begin{proof}
   We mimic the argumentation of Hirzebruch \cite{Hirzebruch}.
   There exists a $(\Z/n\Z)^{\tau-1}$-cover $W$ of $\P^2$ branched with order $n$ along the $d$-arrangement $D$.
   We keep the same notations as in the proof of Theorem \ref{thm:Thm A variant}.
   In particular all maps and varieties defined in the diagram in Figure \ref{dia:diagram}
   remain the same with $Y=\P^2$.
   We compute first $c_{2}(X) = e(X)$. Note that
   $$e\left(X\setminus \bigcup\limits_{P\in\Esing(D)} f^{-1}E_{P}\right) =
     n^{\tau-1}\left(e(\mathbb{P}^{2}) - e(D)\right) + n^{\tau-2}\left(e(D) - e({\rm Sing}(D)\right) + n^{\tau-3}t_{2}.$$
   Simple computations lead to
   $$e\left(X\setminus \bigcup\limits_{P\in\Esing(D)} f^{-1}E_{P}\right) = n^{\tau-1}(3 + (2g-2)\tau +f_{1} -f_{0}) +n^{\tau-2}\left((2-2g)\tau-f_{1}\right) + n^{\tau-3}t_{2},$$
   where $g$ denotes the genus of an irreducible component of $D$, i.e. $g=(d-1)(d-2)/2$.
   Using
   $$\sum_{r \geq 3} n^{\tau-1-r}t_{r}e(F_{P}) = n^{\tau-2} \left(\sum_{r \geq 3} 2t_{r} - \sum_{r \geq 3} rt_{r}\right) + n^{\tau-3} \sum_{r\geq 3} rt_{r}$$
   we obtain
   $$c_{2}(X)/n^{\tau-3} = n^{2}(3+ (2g-2)\tau + f_{1}-f_{0}) + 2n\left( (1-g)\tau+ f_{0}-f_{1}\right) + (f_{1}-t_{2}).$$

   Now we compute $c_{1}^{2}(X) = K_{X}^2$.  From the diagram in Figure \ref{dia:diagram} with $Y=\P^2$ we read off
   that $K_{X} = f^{*}K$, where
   \begin{equation}\label{eq:Kprim defined}
      K =
      \pi^{*}(K_{\mathbb{P}^{2}}) + \sum_{P \in \Esing(D)} E_{P} + \frac{n-1}{n} \left(\sum_{P \in \Esing(D)} E_{P} + \pi^{*}(D) - \sum_{P \in \Esing(D)} k_{P} E_{P}\right).
   \end{equation}
   We have
   $$K = \pi^{*}(K_{\mathbb{P}^{2}})+ \frac{n-1}{n}\pi^{*}(D) +  \sum_{P \in \Esing(D)} \left(1 + \frac{n-1}{n}(1-k_{P})\right)E_{P}.$$
   Since $K_{X}^2 = n^{\tau-1}(K)^2$, we obtain
   $$c_{1}^{2}(X)/n^{\tau-3} =
      n^{2}(K)^{2} =
      9n^{2}+ d^{2}\tau^{2}(n-1)^2-6d\tau n(n-1)-\sum_{r \geq 3}t_{r}\left(n^{2} + (n-1)^{2}(1-r)^{2} + 2n(n-1)(1-r)\right).$$

   We postpone the proof that $X$ is a surface of general type until Lemma \ref{lem:X gen type}.
   Taking this for granted and fixing $n=3$ we apply on $X$
   the Miyaoka-Yau inequality which gives
   $$36(g-1)\tau + 36d\tau - 4d^{2}\tau + 16f_{0} - 4f_{1} - 4t_{2}\geq 0.$$
   Here a side comment is due. Our choice of $n=3$ is a little bit ambiguous. In fact one could work with different values of $n$
   and obtain mutations of inequalities \eqref{degreed} and \eqref{degreD}. These inequalities obtained with various values of $n$
   are hard to compare. Our choice seems asymptotically right and certainly sufficient in order to derive
   Corollary \ref{cor:deg d HC} we do not dwell further on this issue.

   Coming back to the main course of the proof and
   expressing $g$ in terms of $d$, we obtain the following Hirzebruch-type inequality for $d$-arrangements
\begin{equation}
\label{degreed}
   \frac92(d^{2}-3d)\tau + 9d\tau - d^{2}\tau -  t_{2} = \frac72 d^{2}\tau -\frac92 d\tau - t_{2} \geq \sum_{r \geq 2} (r-4)t_{r}.
\end{equation}
   For $h(\P^2;D)$ we have
   $$h(\P^2;D) = \frac{ d^{2}\tau^2- \sum_{r \geq 2} r^{2}t_{r} }{ f_{0} } = \frac{ d^{2}\tau^{2} - f_{2} }{ f_{0} }=\frac{ d^{2}\tau - f_{1} }{f_{0}},$$
   where the last equality follows from $d^{2}\tau^2 -d^{2}\tau = f_{2} - f_{1}$.
   From (\ref{degreed}) we derive that
   $$-f_{1} \geq -4f_{0}-\frac72 d^{2}\tau+\frac92 d\tau+t_{2}$$
   and then
\begin{equation}\label{degreD}
   h(\P^2;D) \geq -4 + \frac{-(5/2)d^{2}\tau+(9/2)d\tau+t_{2}}{f_{0}} \geq -4 + \frac{-(5/2)d^{2}\tau+(9/2)d\tau}{f_{0}},
\end{equation}
   which completes the proof.
\end{proof}
   In order to pass to degree $d$ Harbourne constants, we need to
   get rid of $\tau$ and $f_0$ in \eqref{degreD}.
\begin{lemma}[The number of singular points in a $d$-arrangement]\label{lem:s tau}
   Let $D=\sum C_i$ be a transversal arrangement of $\tau\geq 2$ degree $d$ curves $C_i$ in $\P^2$
   such that $t_\tau=0$. Then $s=s(D)\geq\tau$.
\end{lemma}
\proof
   First we claim that each curve $C_i$ contains at least $d^2+1$ intersection
   points with other curves in the arrangement. Indeed, if not, then by the
   transversality assumption it contains exactly $d^2$ intersection points.
   But this implies that all $\tau$ curves $C_j$ meet exactly in these $d^2$ points
   contradicting the assumption $t_\tau=0$. Let $f:Y\to\P^2$ be the blow up
   of all $s$ singular points of $D$. Then the Picard number of $Y$ is $s+1$.
   On the other hand, the proper transforms $\wtilde{C_1},\ldots,\wtilde{C_\tau}$
   are disjoint curves of self-intersection less or equal to $ d^2-(d^2+1)=-1$ on $Y$. By the Hodge Index Theorem we have then
   $s\geq \tau$ as asserted.
\endproof
   Now we are in the position to prove Corollary \ref{cor:deg d HC}.
\proof
   It is easy to observe that in order to find a lower bound for \eqref{degreD} one needs to find an effective bound for $f_{0}$ and then by Lemma \ref{lem:s tau} we get the desired inequality.
\endproof
   We conclude this section with the following Lemma.
\begin{lemma}[The Kodaira dimension of the divisor $K$]\label{lem:X gen type}
   For
   $d \geq 3$, $n\geq 2$, $\tau\geq 4$ and $t_{\tau}=0$ the
   divisor $K$ defined in \eqref{eq:Kprim defined} is big and nef.
\end{lemma}
\proof
   We argue along the lines of \cite[Section 2.3]{LT}.
   We want first to show that there is a way to write $K$ as an effective $\Q$-divisor. From \eqref{eq:Kprim defined}
   we have
   \begin{equation}\label{eq:K new form}
      K=\pi^*\left(-\frac1d(C_1+C_2+C_3)\right)+\frac{2n-1}{n}\sum E_P+\frac{n-1}{n}\sum\wtilde{C_i},
   \end{equation}
   where $\wtilde{C_i}=\pi^*C_i-\sum\limits_{P\in(C_i\cap\Esing(D))} E_P$ is the proper transform of $C_i$
   under $\pi$. This divisor can be written as
   $$K=\sum a_i\wtilde{C_i}+\sum b_PE_P$$
   with positive coefficients
   $$a_i\geq \frac{n-1}{n}-\frac1d>0\;\;\mbox{ and }\;\; b_P\geq \frac{2n-1}{n}-\frac3d>0.$$
   Thus in order to check that $K$ is nef it suffices to check its intersection
   with curves in its support.
   For $E_P$ we have from \eqref{eq:K new form}
   $$K.E_P=-\frac{2n-1}{n}+\frac{n-1}{n}k_P\geq \frac{n-2}{n}\geq 0.$$
   For the intersection with $\wtilde{C}:=\wtilde{C_i}$ for some $i\in\left\{1,\ldots,\tau\right\}$ it is more convenient to pass
   to the numerical equivalence classes:
   $$K\equiv \left(\tau d\frac{n-1}n-3\right)H+\sum\left(\frac{2n-1}{n}-k_P\frac{n-1}{n}\right)E_P\;\;\mbox{ and }\;\;
      \wtilde{C}\equiv dH-\sum_{P\in(C\cap\Esing(D))}E_P,$$
   where $H=\pi^*(\calo_{\P^2}(1))$.
   We obtain
   \begin{equation}\label{eq:loc1}
      K.\wtilde{C}=\tau d^2\frac{n-1}{n}-3d+\sum_{P\in(C\cap\Esing(D))}\left(1+\frac{n-1}{n}(1-k_P)\right).
   \end{equation}
   Now, the last summand can be written as
   $$\#\left\{\Esing(D)\cap C\right\}-\frac{n-1}{n}\sum_{P\in(C\cap\Esing(D))}(k_P-1).$$
   Recalling the following equality coming from counting incidences with the component $C$ in two ways
   \begin{equation}
      \sum_{P\in(C\cap \Sing(D))} (k_P-1)=d^2(\tau-1)
   \end{equation}
   and plugging it into \eqref{eq:loc1} we obtain
   $$K.\wtilde{C}=\frac{n-1}{n}d^2+\frac{n-1}{n}\#\left\{P\in C:\; k_P=2\right\}+\#\left\{P\in C:\; k_P\geq 3\right\}-3d.$$
   Now, as in the proof of Lemma \ref{lem:s tau} we have $\#\left\{P\in C:\; k_P\geq 2\right\}\geq (d^2+1)$
   so that the last two summand can be bounded from below by
   $\frac{n-1}{n}(d^2+1)$.
   Rearranging the terms we get finally
   $$K.\wtilde{C}\geq \frac{2n-2}{n}d^2-3d+\frac{n-1}{n}.$$
   The expression on the right is positive for $d\geq 3$ and $n\geq 2$.
   This finishes the proof that $K$ is nef.

   In order to show that $K$ is also big it suffices to check that
   its self-intersection is positive. We omit an easy calculation.
\endproof
\paragraph*{\emph{Acknowledgement.}}
 The first author was partially supported by National Science Centre Poland Grant 2014/15/N/ST1/02102 and the project was conduct when he was a member of SFB/TR 45 \emph{Periods, moduli spaces and arithmetic of algebraic varieties}.
  The last author was partially supported by National Science Centre, Poland, grant 2014/15/B/ST1/02197.
  The final version of this work was written down while the last author visited
  the University of Mainz. A generous support of the SFB/TR 45 \emph{Periods, moduli spaces and arithmetic of algebraic varieties}
  is kindly acknowledged.\\
  Finally we would like to thank the referee for many valuable comments which led to improvements
  in the readability of our manuscript.
%*****************************************************************************

%***************************************************************************** % Addresses

\medskip
   Piotr Pokora,
   Instytut Matematyki,
   Pedagogical University of Cracow,
   Podchor\c a\.zych 2,
   PL-30-084 Krak\'ow, Poland.

Current Address:
    Institut f\"ur Algebraische Geometrie,
    Leibniz Universit\"at Hannover,
    Welfengarten 1,
    D-30167 Hannover, Germany. \\
\nopagebreak
   \textit{E-mail address:} \texttt{piotrpkr@gmail.com, pokora@math.uni-hannover.de}
\medskip

   Tomasz Szemberg
   Instytut Matematyki,
   Pedagogical University of Cracow,
   Podchor\c a\.zych 2,
   PL-30-084 Krak\'ow, Poland.

\nopagebreak
   \textit{E-mail address:} \texttt{tomasz.szemberg@gmail.com}
	
\medskip

  Xavier Roulleau
  Laboratoire de Math\'ematiques et Applications,
	Universit\'e de Poitiers, UMR CNRS 7348,
  T\'el\'eport 2 - BP 30179 - 86962 Futuroscope Chasseneuil, France.

\nopagebreak
   \textit{E-mail address:} \texttt{xavier.roulleau@math.univ-poitiers.fr}

%*****************************************************************************

\end{document}